\documentclass[12pt]{article}

\usepackage[margin=1in]{geometry}
\usepackage{amsthm}
\usepackage{amssymb}
\usepackage{amsmath}
\usepackage{graphicx}
\usepackage{mathdots}
\usepackage{mathtools}
\usepackage{url}
\usepackage{color}
\usepackage{framed}

\newtheorem{theorem}{Theorem}%[section]

\newtheorem{lemma}[theorem]{Lemma}
\newtheorem{proposition}[theorem]{Proposition}
\newtheorem{conjecture}[theorem]{Conjecture}

\theoremstyle{definition}

\title{On the concentration of Gaussian Cayley matrices}

\author{Afonso~S.~Bandeira\footnote{Department of Mathematics, ETH Z\"{u}rich} \qquad Dmitriy Kunisky\footnote{Department of Computer Science, Yale University} \medskip \\ Dustin~G.~Mixon\footnote{Department of Mathematics, The Ohio State University} \footnote{Translational Data Analytics Institute, The Ohio State University} \qquad Xinmeng Zeng\footnotemark[1]}
\date{}

\begin{document}
\maketitle

\begin{abstract}
Given a finite group, we study the Gaussian series of the matrices in the image of its left regular representation.
We propose such random matrices as a benchmark for improvements to the noncommutative Khintchine inequality, and we highlight an application to the matrix Spencer conjecture.
\end{abstract}

\section{Introduction}

Given $A_1,\ldots,A_n\in\mathbb{C}^{d\times d}$, the corresponding \textit{Gaussian matrix series} is given by
\[
X:=\sum_{i=1}^n x_i A_i,
\]
where $\{x_i\}_{i=1}^n$ are independent standard Gaussian variables.
We are interested how the tuple $\{A_i\}_{i=1}^n$ of coefficient matrices influences the expected spectral norm $\mathbb{E}\|X\|$.
To this end, the noncommutative Khintchine inequality of Lust-Piquard and Pisier~\cite{LustPiquardP:91,Pisier:03} delivers an upper bound that is sharp for some commutative tuples, but suboptimal for many noncommutative tuples.
Certain quantifiable notions of noncommutativity have been treated by the improved bounds of Tropp~\cite{Tropp:18} and of Bandeira, Boedihardjo, and van Handel~\cite{BandeiraBvH:21}, but many instances are yet to be understood.
We shine a light on this gap by introducing a noteworthy family of random matrices.

Given a finite group $G$ of order $n$, consider the \textit{left regular representation} $\rho\colon G\to\operatorname{GL}(n)$ defined by $\rho(g)e_h:=e_{gh}$ and extending linearly.
The \textit{Gaussian Cayley matrix} is given by
\[
X_G
:=\sum_{g\in G}x_g\rho(g),
\]
where $\{x_g\}_{g\in G}$ are independent standard Gaussian variables.
By Cayley's theorem, the coefficient matrices $\{\rho(g)\}_{g\in G}$ form a group that is isomorphic to $G$.
These coefficient matrices are special because their noncommutativity can be expressed in the language of group theory and representation theory.
This expressivity makes Gaussian Cayley matrices a particularly revealing class of Gaussian matrix series, as illustrated by our main result:

\begin{theorem}
\label{thm.main}
Suppose $G$ is a finite group of order $n$, let $\Sigma$ denote the set of isomorphism classes of irreducible representations of $G$, let $d_\pi$ denote the degree of $\pi\in\Sigma$, and put
\[
m(G)
:=\inf_{s\geq0}\bigg( s + \sum_{\pi \in \Sigma}\frac{1}{\sqrt{d_\pi}}e^{-d_\pi s^2/2} \bigg).
\]
Then\footnote{We write $x\lesssim y$ if $x\leq Cy$ for some universal constant $C>0$, and $x\asymp y$ if both $x\lesssim y$ and $y\lesssim x$.}
\[
\sqrt{n}
\lesssim\mathbb{E}\|X_G\|
\lesssim\sqrt{n}\cdot m(G)
\lesssim\sqrt{n\log n}.
\]
Furthermore:
\begin{itemize}
\item[(a)]
If $G$ is abelian, then $\mathbb{E}\|X_G\|\asymp\sqrt{n}\cdot m(G)\asymp\sqrt{n\log n}$.
\item[(b)]
If $G$ is simple and nonabelian, then $\mathbb{E}\|X_G\|\asymp\sqrt{n}\cdot m(G)\asymp\sqrt{n}$.
\end{itemize}
\end{theorem}

As we discuss in the following section, neither the noncommutative Khintchine inequality~\cite{LustPiquardP:91,Pisier:03} nor its existing improvements~\cite{Tropp:18,BandeiraBvH:21} are sharp enough to prove Theorem~\ref{thm.main}.
In Section~3, we prove Theorem~\ref{thm.main} using tools from representation theory.
Section~4 then applies Theorem~\ref{thm.main} to resolve a special case of the matrix Spencer conjecture~\cite{Zouzias:12,Meka:14}.
We conclude in Section~5 with a discussion.

\section{The failure of existing bounds}

We start with the noncommutative Khintchine inequality of Lust-Piquard and Pisier~\cite{LustPiquardP:91,Pisier:03}:

\begin{proposition}
\label{prop.nck}
Given self-adjoint $A_1,\ldots,A_n\in\mathbb{C}^{d\times d}$, the corresponding Gaussian matrix series $X$ satisfies
\[
\sigma(X)
\lesssim \mathbb{E}\|X\|
\lesssim \sigma(X)\log^{1/2} d,
\qquad
\sigma(X):=\bigg\|\sum_{i=1}^n A_i^2\bigg\|^{1/2}.
\]
\end{proposition}

The upper bound in Proposition~\ref{prop.nck} is sharp, for example, when $n=d$ and $A_i=e_ie_i^*$ for each $i\in\{1,\ldots,d\}$.
In this case, $\sigma(X)=1$ and $\|X\|$ is the maximum of $d$ independent half normal variables, the expectation of which is $\Theta(\sqrt{\log d})$; see for example Section~3.3 in~\cite{LedouxT:91}.
The lower bound in Proposition~\ref{prop.nck} is sharp, for example, when $n=\binom{d+1}{2}$ and $d$ of the $A_i$'s take the form $\sqrt{2}e_je_j^*$, while the other $\binom{d}{2}$ $A_i$'s take the form $e_je_k^*+e_ke_j^*$ with $j<k$.
In this case, $\sigma(X)=\sqrt{d}$ and
$X$ is drawn from the Gaussian orthogonal ensemble, which is known to satisfy $\mathbb{E}\|X\|\asymp\sqrt{d}$; see for example Section~2.3 in~\cite{Tao:12}.

Notice that Proposition~\ref{prop.nck} is written in terms of self-adjoint coefficient matrices.
The general setting can be mapped to the self-adjoint case by \textit{dilation}:
\[
\tilde{A}:=\left[\begin{array}{cc}0&A^*\\A&0\end{array}\right].
\]
Observe that $A\in\mathbb{C}^{d\times d}$ implies $\tilde{A}\in\mathbb{C}^{2d\times 2d}$ with
\[
\|\tilde{A}\|^2
=\|\tilde{A}^2\|
=\max\{\|A^*A\|,\|AA^*\|\}
=\|A\|^2.
\]
Furthermore, if $X=\sum_i x_iA_i$, then $\tilde{X}=\sum_i x_i\tilde{A}_i$.

Let's consider Proposition~\ref{prop.nck} in the context of Gaussian Cayley matrices.
Suppose $G$ is a finite group of order $n$.
Then for each $g\in G$, the square of the dilation of $\rho(g)$ is $I_{2n}$, and so $\sigma(\tilde{X}_G)=\|nI_{2n}\|^{1/2}=\sqrt{n}$.
As such, Proposition~\ref{prop.nck} gives
\[
\sqrt{n}\lesssim\mathbb{E}\|X_G\|\lesssim\sqrt{n\log n},
\]
regardless of $G$.
Apparently, parts~(a) and~(b) of Theorem~\ref{thm.main} identify extreme cases of the noncommutative Khintchine inequality.
To prove Theorem~\ref{thm.main}, we need to more carefully account for interactions between coefficient matrices in the Gaussian matrix series.

Tropp~\cite{Tropp:18} obtained the following improvement over the upper bound in Proposition~\ref{prop.nck}:

\begin{proposition}
\label{prop.tropp}
Given self-adjoint $A_1,\ldots,A_n\in\mathbb{C}^{d\times d}$, the corresponding Gaussian matrix series $X$ satisfies
\[
\mathbb{E}\|X\|
\lesssim \sigma(X)\log^{1/4}d+w(X)\log^{1/2}d,
\]
where $\sigma(X)$ is defined in Proposition~\ref{prop.nck} and
\[
w(X):=\max_{Q_1,Q_2,Q_3\in\operatorname{U}(d)}\bigg\|\sum_{i=1}^n\sum_{j=1}^n A_iQ_1A_jQ_2A_iQ_3A_j\bigg\|^{1/4}.
\] 
\end{proposition}

Proposition~3.2 in~\cite{Tropp:18} gives that $w(X)\leq\sigma(X)$, from which it follows that Proposition~\ref{prop.tropp} improves upon Proposition~\ref{prop.nck}.
This improvement is insufficient to prove Theorem~\ref{thm.main}:

\begin{lemma}
Suppose $G$ is a finite group of order $n$.
Then $w(\tilde{X}_G)=\sqrt{n}$.
\end{lemma}

\begin{proof}
For each $g\in G$, let $A_g$ denote the dilation of $\rho(g)$.
Then
\[
(A_gA_h)^2
=\left[\begin{array}{cc}\rho(g^{-1}h)&0\\0&\rho(gh^{-1})\end{array}\right]^2
=\left[\begin{array}{cc}\rho((g^{-1}h)^2)\\0&\rho((gh^{-1})^2)\end{array}\right],
\]
and so
\[
\sum_{g,h\in G}(A_gA_h)^2
=nI_2\otimes\sum_{g\in G}\rho(g^2).
\]
Observe that $S:=\sum_{g\in G}\rho(g^2)$ is symmetric:
\[
S^\top
=\sum_{g\in G}\rho(g^2)^\top
=\sum_{g\in G}\rho((g^2)^{-1})
=\sum_{g\in G}\rho((g^{-1})^2)
=S,
\]
and so $\|S\|$ equals the largest eigenvalue of $S$ (in magnitude).
Considering
\[
S\mathbf{1}
=\sum_{g\in G}\rho(g^2)\mathbf{1}
=\sum_{g\in G}\mathbf{1}
=n\mathbf{1},
\]
we have $\|S\|\geq n$, and so 
\[
w(\tilde{X}_G)
\geq\bigg\|\sum_{g,h\in G}(A_gA_h)^2\bigg\|^{1/4}
=\| nI_2\otimes S\|^{1/4}
=n^{1/4}\|S\|^{1/4}
\geq \sqrt{n}.
\]
The opposite bound $w(\tilde{X}_G)\leq\sigma(\tilde{X}_G)=\sqrt{n}$ implies equality.
\end{proof}

We conclude this section by discussing the following improvement by Bandeira, Boedihardjo, and van Handel~\cite{BandeiraBvH:21} over Proposition~\ref{prop.nck}:

\begin{proposition}
\label{prop.BBvH}
Given self-adjoint $A_1,\ldots,A_n\in\mathbb{C}^{d\times d}$, the corresponding Gaussian matrix series $X$ satisfies
\[
\mathbb{E}\|X\|
\lesssim \sigma(X)+v(X)^{1/2}\sigma(X)^{1/2}\log^{3/4}d,
\]
where $\sigma(X)$ is defined in Proposition~\ref{prop.nck} and $v(X)^2$ is the spectral norm of the covariance matrix $\operatorname{Cov}(X)$ of the entries of $X$, defined by $\operatorname{Cov}(X)_{ij,kl}:=\mathbb{E}[X_{ij}\overline{X_{kl}}]$. 
\end{proposition}

Proposition~\ref{prop.BBvH} improves upon Proposition~\ref{prop.nck} in the regime where $v(X)$ is much smaller than $\sigma(X)/\sqrt{\log d}$, e.g., when $X$ is drawn from the Gaussian orthogonal ensemble.
However, the Gaussian Cayley matrix does not reside in this regime.
To see this, suppose $G$ is a finite group of order $n$.
Then $\tilde{X}_G$ has a total of $4n^2$ entries, $2n^2$ of which are identically zero, and the other $2n^2$ can be partitioned into $n$ independent batches of size $2n$, where the entries in each batch are identical standard Gaussian variables.
As such, $\operatorname{Cov}(\tilde{X}_G)$ is block diagonal with $n$ different $2n\times 2n$ all-ones blocks, and otherwise zero.
It follows that $v(\tilde{X}_G)=\|\operatorname{Cov}(\tilde{X}_G)\|^{1/2}=\sqrt{2n}$.

\section{Sharp bounds using representation theory}

In this section, we use representation theory to prove Theorem~\ref{thm.main}.
For this approach, it is more natural to consider the \textit{complex Gaussian Cayley matrix} defined by
\[
Z_G
:=\sum_{g\in G}z_g\rho(g),
\]
where $\{z_g\}_{g\in G}$ are independent \textit{complex} Gaussian variables such that the real and imaginary parts of each $z_g$ are independent standard Gaussian variables.
The following lemma establishes that we only lose constants by studying this random matrix:

\begin{lemma}
\label{lem.real vs complex}
Suppose $G$ is a finite group.
Then $\mathbb{E}\|X_G\|\leq \mathbb{E}\|Z_G\|\leq 2\mathbb{E}\|X_G\|$.
\end{lemma}

\begin{proof}
Observe that $A:=\operatorname{Re}Z_G$ and $B:=\operatorname{Im}Z_G$ each have the same distribution as $X_G$.
For every $u\in\mathbb{R}^n$, it holds that $\|Z_Gu\|^2=\|Au\|^2+\|Bu\|^2\geq \|Au\|^2$, and so
\[
\|Z_G\|
=\sup_{\substack{u\in\mathbb{C}^n\\\|u\|=1}}\|Z_Gu\|
\geq\sup_{\substack{u\in\mathbb{R}^n\\\|u\|=1}}\|Z_Gu\|
\geq\sup_{\substack{u\in\mathbb{R}^n\\\|u\|=1}}\|Au\|
=\|A\|.
\]
Taking the expectation then gives the first inequality.
The second inequality follows from taking the expectation of $\|Z_G\|=\|A+\mathrm{i}B\|\leq\|A\|+\|B\|$.
\end{proof}

Given a finite group $G$, let $\Sigma$ denote the set of isomorphism classes of irreducible representations of $G$, and let $d_\pi$ denote the degree of $\pi\in\Sigma$.

\begin{lemma}
\label{lem.diagonalization}
Suppose $G$ is a finite group of order $n$.
There exists a deterministic $U\in\operatorname{U}(n)$ such that the complex Gaussian Cayley matrix $Z_G$ has the same distribution as the random block-diagonalized matrix
\[
U\bigg[\bigoplus_{\pi\in\Sigma} (I_{d_\pi}\otimes \sqrt{\tfrac{n}{d_\pi}}Z_\pi)\bigg]U^*,
\]
where $\{Z_\pi\}_{\pi\in\Sigma}$ denote independent random matrices with $Z_\pi\in\mathbb{C}^{d_\pi\times d_\pi}$ having complex Gaussian entries whose real and imaginary parts are independent standard Gaussian variables.
\end{lemma}

\begin{proof}
For each $\pi\in\Sigma$, select $\lambda_\pi\in\pi$ such that $\lambda_\pi\colon G\to\operatorname{U}(d_\pi)$.
By Peter--Weyl, there exists $U\in\operatorname{U}(n)$ such that
\[
\rho(g)=U\bigg[\bigoplus_{\pi\in\Sigma} (I_{d_\pi}\otimes \lambda_\pi(g))\bigg]U^*,
\qquad
g\in G.
\]
By Schur's orthogonality relations, the linear map defined by
\[
\{a_g\}_{g\in G}
\mapsto
\bigg\{\sqrt{\tfrac{d_\pi}{n}}\sum_{g\in G}a_g\lambda_\pi(g)_{ij}\bigg\}_{\pi\in\Sigma,i,j\in[d_\pi]}
\]
is unitary.
It follows from the rotation invariance of the spherical Gaussian that the tuple $\{\sqrt{d_\pi/n}\sum_{g\in G}z_g\lambda_\pi(g)\}_{\pi\in\Sigma}$ of random matrices has the same distribution as $\{Z_\pi\}_{\pi\in\Sigma}$.
\end{proof}

\begin{lemma}
\label{lem.bound for pseudo-quasirandom}
Suppose $G$ is a finite group of order $n$.
Then
\[
\mathbb{E}\|Z_G\|
\lesssim \sqrt{n}\cdot m(G),
\qquad
m(G):=\inf_{s\geq0}\bigg( s + \sum_{\pi \in \Sigma}\frac{1}{\sqrt{d_\pi}}e^{-d_\pi s^2/2} \bigg).
\]
Furthermore, $1\lesssim m(G) \lesssim \sqrt{\log n}$.
\end{lemma}

\begin{proof}
First, Lemma~\ref{lem.diagonalization} gives
\begin{equation}
\label{eq.main bound 1}
\tfrac{1}{\sqrt{n}}\mathbb{E}\|Z_G\|
=\mathbb{E}\max_{\pi\in\Sigma}\|\tfrac{1}{\sqrt{d_\pi}}Z_\pi\|
=\int_0^\infty\mathbb{P}\bigg\{\max_{\pi\in\Sigma}\|\tfrac{1}{\sqrt{d_\pi}}Z_\pi\|>t\bigg\}dt.
\end{equation}
By Gordon's theorem~\cite{Vershynin:12}, there is a universal constant $C>0$ such that $\mathbb{E}\|Z_\pi\|\leq C\sqrt{d_\pi}$ for all $\pi$.
Fix $s\geq0$ to be selected later.
We truncate the integral in \eqref{eq.main bound 1} at $C+s$ to obtain
\begin{equation}
\label{eq.main bound 2}
\tfrac{1}{\sqrt{n}}\mathbb{E}\|Z_G\|
\leq C+s+\int_{C+s}^\infty\mathbb{P}\bigg\{\max_{\pi\in\Sigma}\|\tfrac{1}{\sqrt{d_\pi}}Z_\pi\|>t\bigg\}dt
=:C+s+I.
\end{equation}
Next, we apply the union bound to the integrand in \eqref{eq.main bound 2}:
\begin{equation}
\label{eq.main bound 3}
I
\leq \int_{C+s}^\infty\sum_{\pi\in\Sigma}\mathbb{P}\bigg\{\|\tfrac{1}{\sqrt{d_\pi}}Z_\pi\|>t\bigg\}dt
=\sum_{\pi\in\Sigma}\int_s^\infty\mathbb{P}\bigg\{\|Z_\pi\|-C\sqrt{d_\pi}>t\sqrt{d_\pi}\bigg\}dt.
\end{equation}
The mapping $(\mathbb{R}^{d_\pi\times d_\pi})^2\to\mathbb{R}$ defined by $(X,Y)\mapsto\|X+\mathrm{i}Y\|$ is $1$-Lipschitz with respect to the mixed Frobenius-$2$ norm, and so Gaussian concentration gives
\[
\mathbb{P}\Big\{\|Z_\pi\|-\mathbb{E}\|Z_\pi\|\geq t\Big\}\leq e^{-t^2/2},
\qquad
t\geq0.
\]
We use this to continue our bound on \eqref{eq.main bound 3}:
\[
I
\leq\sum_{\pi\in\Sigma}\int_s^\infty\mathbb{P}\bigg\{\|Z_\pi\|-\mathbb{E}\|Z_\pi\|>t\sqrt{d_\pi}\bigg\}dt
\leq\sum_{\pi\in\Sigma}\int_s^\infty e^{-d_\pi t^2/2}dt
\leq\sqrt{2\pi}\sum_{\pi\in\Sigma}\frac{1}{\sqrt{d_\pi}}e^{-d_\pi s^2/2}.
\]
Putting everything together, we therefore have
\[
\tfrac{1}{\sqrt{n}}\mathbb{E}\|Z_G\|
\leq C+s+\sqrt{2\pi}\sum_{\pi\in\Sigma}\frac{1}{\sqrt{d_\pi}}e^{-d_\pi s^2/2}
\lesssim C+m(G),
\]
where the last step selects an optimal $s\geq0$.
It remains to show that $1\lesssim m(G)\lesssim\sqrt{\log n}$.
For the first inequality, isolating the trivial representation gives
\[
s + \sum_{\pi \in \Sigma}\frac{1}{\sqrt{d_\pi}}e^{-d_\pi s^2/2}
\geq s+e^{-s^2/2}
\geq e^{-1/2}
\]
for all $s\geq0$, where the last step follows by casing on whether $s\leq1$.
For the second inequality, the bounds $d_\pi\geq1$ and $|\Sigma|\leq n$ together give
\[
s + \sum_{\pi \in \Sigma}\frac{1}{\sqrt{d_\pi}}e^{-d_\pi s^2/2}
\leq s + ne^{-s^2/2}.
\]
The infimum of the right-hand side is at most $\sqrt{2\log n}+1$ by taking $s:=\sqrt{2\log n}$.
\end{proof}

If $G$ is an abelian group of order $n$, then $|\Sigma|=n$ and $d_\pi=1$ for every $\pi\in\Sigma$, and so $m(G)\asymp\sqrt{\log n}$, i.e., Lemma~\ref{lem.bound for pseudo-quasirandom} does not improve over the noncommutative Khintchine inequality in such cases.
Next, fix $k\in\mathbb{N}$ and $\epsilon>0$, and consider all groups $G$ such that
\[
|\{\pi\in\Sigma:d_\pi<\epsilon\log|G|\}|\leq k.
\]
Then taking $s:=\sqrt{2/\epsilon}$ reveals that $m(G)$ is bounded:
\[
m(G)
\leq s+\sum_{\pi \in \Sigma}\frac{1}{\sqrt{d_\pi}}e^{-d_\pi s^2/2}
\leq s+k+\frac{|\Sigma|}{\sqrt{\epsilon\log|G|}}\cdot|G|^{-\epsilon s^2/2}
\leq \sqrt{2/\epsilon}+k+1/\sqrt{\epsilon}.
\]
By the following lemma, this occurs whenever $G$ is a nonabelian finite simple group.

\begin{lemma}
\label{lem.repn sizes for simple nonabelian}
There exist universal constants $k\in\mathbb{N}$ and $\epsilon>0$ for which the following holds:
For every nonabelian finite simple group $G$, it holds that $|\{\pi\in\Sigma:d_\pi<\epsilon\log|G|\}|\leq k$.
\end{lemma}

\begin{proof}
We argue in cases using the classification of finite simple groups~\cite{Aschbacher:04}.

First, we treat the alternating groups.
For all sufficiently large $n$, it holds that $S_n$ has four irreducible representations of degree smaller than $n(n-3)/2$, namely, the trivial representation of degree $1$ (corresponding to the partition $(n)$), the standard representation of degree $n-1$ (corresponding to the partition $(n-1,1)$), and their conjugates (corresponding to the transposed Young diagrams); see~\cite{James:83}, for example.
By Proposition~5.1 in~\cite{FultonH:04}, the restrictions of these representations to $A_n$ give the only two irreducible representations (up to isomorphism) of degree smaller than $n(n-3)/2$.
Since $n(n-3)/2\geq\log|A_n|$, the claim holds for all sufficiently large alternating groups.

Next, we treat the classical groups $A_n(q)$, $B_n(q)$, $C_n(q)$, $D_n(q)$, $^2A_n(q^2)$, and $^2D_n(q^2)$.
Letting $G$ denote any of these groups, it holds that $\log|G|\leq Cn^2\log q$, where $C>0$ is a universal constant.
The third table in \cite{Collins:08} gives that every nontrivial representation of each of these groups has degree at least $cq^n$, where $c>0$ is a universal constant.
One may verify that $q^n\geq n^2\log q$ for all $n,q\in\mathbb{N}$, and so the claim holds for all classical groups.

Next, consider the exceptional groups of Lie type, namely, $E_6(q)$, $E_7(q)$, $E_8(q)$, $F_4(q)$, $G_2(q)$, $^2E_6(q^2)$, $^3D_4(q^3)$, $^2B_2(q)$, $^2F_4(q)$, and $^2G_2(q)$.
Letting $G$ denote any of these groups, it holds that $\log|G|\leq C\log q$, where $C>0$ is a universal constant.
The fourth table in \cite{Collins:08} gives that every nontrivial representation of each of these groups has degree at least $q$, and so the claim holds for all exceptional groups of Lie type.

Finitely many finite simple groups remain, and so we are done.
\end{proof}

We note that the alternating groups prevent us from taking $k=1$ in Lemma~\ref{lem.repn sizes for simple nonabelian}.
This was observed by Gowers in his study of quasirandom groups~\cite{Gowers:08}, that is, sequences of groups for which the minimum degree of the nontrivial representations is unbounded.
In~\cite{Gowers:08}, Gowers presents an elementary proof that every nontrivial representation of every nonabelian finite simple group $G$ has degree at least $\frac{1}{2}\sqrt{\log|G|}$, and it would be nice to obtain a similarly elementary proof of Lemma~\ref{lem.repn sizes for simple nonabelian}.
We are now ready to prove our main result.

\begin{proof}[Proof of Theorem~\ref{thm.main}]
The inequalities follow from combining Proposition~\ref{prop.nck} with Lemmas~\ref{lem.real vs complex} and~\ref{lem.bound for pseudo-quasirandom}.
For (a), suppose $G$ is abelian.
Then by Lemma~\ref{lem.diagonalization}, $\|Z_G\|$ is $\sqrt{n}$ times the maximum of the absolute values of $n$ independent standard complex Gaussian variables.
The expectation of this maximum is $\Theta(\sqrt{\log n})$; see for example Section~3.3 in~\cite{LedouxT:91}.
Combining with Lemma~\ref{lem.real vs complex} then gives $\|X_G\|\asymp\|Z_G\|\asymp\sqrt{n\log n}$.
For (b), suppose $G$ is simple and nonabelian.
Then Lemmas \ref{lem.bound for pseudo-quasirandom} and~\ref{lem.repn sizes for simple nonabelian} together imply that $m(G)\asymp1$.
\end{proof}

\section{An application to the matrix Spencer conjecture}

Spencer's theorem~\cite{Spencer:85} is a celebrated result in discrepancy theory:

\begin{proposition}
\label{prop.spencer}
For every $a_1,\ldots,a_n\in\mathbb{R}^d$ such that $\|a_i\|_\infty\leq1$ for all $i\in\{1,\ldots,n\}$, there exist $\epsilon_1,\ldots,\epsilon_n\in\{-1,1\}$ such that
\[
\bigg\|\sum_{i=1}^n\epsilon_ia_i\bigg\|_\infty
\lesssim \sqrt{n}\cdot\max\Big\{1,\sqrt{\log(d/n)}\Big\}.
\]
\end{proposition}

The \textit{matrix Spencer conjecture}~\cite{Zouzias:12,Meka:14} is a noncommutative analogue of Proposition~\ref{prop.spencer}:

\begin{conjecture}
For every $A_1,\ldots,A_n\in\mathbb{R}^{d\times d}$ such that $\|A_i\|\leq1$ for all $i\in\{1,\ldots,n\}$, there exist $\epsilon_1,\ldots,\epsilon_n\in\{-1,1\}$ such that
\[
\bigg\|\sum_{i=1}^n\epsilon_iA_i\bigg\|
\lesssim \sqrt{n}\cdot\max\Big\{1,\sqrt{\log(d/n)}\Big\}.
\]
\end{conjecture}

Various cases of the matrix Spencer conjecture have been treated in~\cite{LevyRR:17,DadushJR:22,HopkinsRS:22,BansalJM:22}.
In what follows, we leverage Theorem~\ref{thm.main} to resolve yet another special case:

\begin{theorem}
\label{thm.matrix spencer group case}
For every abelian or simple group $G$ of order $n$, there exists $\epsilon_g\in\{-1,1\}$ for each $g\in G$ such that
\[
\bigg\|\sum_{g\in G}\epsilon_g\rho(g)\bigg\|
\lesssim \sqrt{n}.
\]
\end{theorem}

\begin{proof}
We case on whether $G$ is abelian or simple and nonabelian.

\medskip

\noindent
\textbf{Case I:} $G$ is abelian.
Then there exists $U\in\operatorname{U}(n)$ such that
\[
\rho(g)
=U\cdot\operatorname{diag}\{\chi(g)\}_{\chi\in\hat{G}} \cdot U^{-1}
\]
for every $g\in G$.
Here, $\hat{G}$ denotes the group of characters $\chi\colon G\to\mathbb{C}^\times$.
As such,
\[
\bigg\|\sum_{g\in G}\epsilon_g\rho(g)\bigg\|
=\bigg\|\sum_{g\in G}\epsilon_g\operatorname{diag}\{\chi(g)\}_{\chi\in\hat{G}}\bigg\|
=\bigg\|\sum_{g\in G}\epsilon_ga_g\bigg\|_\infty,
\]
where $a_g\in\mathbb{C}^n$ has coordinates indexed by $\hat{G}$ and is defined by $(a_g)_\chi:=\chi(g)$.
Given $z\in \mathbb{C}^n$, denote $\hat{z}:=(\operatorname{Re}z,\operatorname{Im}z)\in\mathbb{R}^{2n}$.
Then $\|\hat{z}\|_\infty\leq\|z\|_\infty\leq\sqrt{2}\|\hat{z}\|_\infty$.
In particular, $\|\hat{a}_g\|_\infty\leq1$ for every $g\in G$, and so Proposition~\ref{prop.spencer} delivers $\epsilon_g\in\{-1,1\}$ for each $g\in G$ such that
\[
\bigg\|\sum_{g\in G}\epsilon_g\rho(g)\bigg\|
=\bigg\|\sum_{g\in G}\epsilon_ga_g\bigg\|_\infty
\leq\sqrt{2}\cdot\bigg\|\sum_{g\in G}\epsilon_g\hat{a}_g\bigg\|_\infty
\lesssim\sqrt{n}.
\]

\medskip

\noindent
\textbf{Case II:} $G$ is simple and nonabelian.
Let $\{r_g\}_{g\in G}$ denote independent Rademacher variables, and let $\{x_g\}_{g\in G}$ denote independent standard Gaussian variables.
Then $x_g=\operatorname{sgn}(x_g)|x_g|$ has the same distribution as $r_g|x_g|$, and so Jensen's inequality gives
\[
\mathbb{E}\|X_G\|
=\mathbb{E}_r\mathbb{E}_x\bigg\|\sum_{g\in G}r_g|x_g|\rho(g)\bigg\|
\geq \mathbb{E}_r\bigg\|\sum_{g\in G}r_g\big(\mathbb{E}|x_g|\big)\rho(g)\bigg\|
=\sqrt{\tfrac{2}{\pi}}\cdot\mathbb{E}\bigg\|\sum_{g\in G}r_g\rho(g)\bigg\|.
\]
We rearrange and apply Theorem~\ref{thm.main}(b) to obtain
\[
\min_\epsilon\bigg\|\sum_{g\in G}\epsilon_g\rho(g)\bigg\|
\leq\mathbb{E}\bigg\|\sum_{g\in G}r_g\rho(g)\bigg\|
=\sqrt{\tfrac{\pi}{2}}\cdot\mathbb{E}\|X_G\|
\asymp\sqrt{n}.
\qedhere
\]
\end{proof}

\section{Discussion}

In this paper, we introduced a family of random matrices that traverses the extremes of the noncommutative Khintchine inequality, while being amenable to analysis by representation theory.
In the language of Lemma~\ref{lem.diagonalization}, we have
\[
\tfrac{1}{\sqrt{n}}\mathbb{E}\|X_G\|
\asymp\mathbb{E}\max_{\pi\in\Sigma}\|\tfrac{1}{\sqrt{d_\pi}}Z_\pi\|
\lesssim m(G),
\]
and furthermore, the bound $m(G)$ is tight when $G$ is abelian or simple.

It would be interesting to prove Theorem~\ref{thm.matrix spencer group case} for more general groups, perhaps by somehow interpolating between the abelian and simple nonabelian cases.
The authors were able to adapt Gluskin's volume argument~\cite{Gluskin:89} by estimating the tail probability of the spectral norm of the block-diagonalized random matrix in Lemma~\ref{lem.diagonalization}, and then using the Gaussian correlation inequality in place of Sidak's lemma to produce a partial coloring, but we were not able to iterate the procedure.

As a generalization of our setting, one might consider any $A_1,\ldots,A_n\in\mathbb{R}^{d\times d}$ such that each entry of each matrix is either $0$ or $1$, and $\sum_{i=1}^nA_i$ is the all-ones matrix.
For example, this occurs whenever $\{A_i\}_{i=1}^n$ are the adjacency matrices of an association scheme.
The Gaussian series of these matrices can be thought of as a ``patterned Gaussian matrix'' (cf.\ Section~3.2.1 in~\cite{BandeiraBvH:21}), in which the entries can be partitioned into independent batches, with entries in each batch being identical.
Such examples are also interesting in the context of the noncommutative Khinchine inequality, but the group case is particularly tractable thanks to its algebraic structure.

\section*{Acknowledgments}

This work was initiated at the 2021 Oberwolfach workshop on Applied Harmonic Analysis and Data Science.
The first two authors thank Jess Banks for insightful discussions on the subject of this paper.
This work was partially supported by ONR Award N00014-20-1-2335 and a Simons Investigator Award from the Simons Foundation to Daniel Spielman.

\end{document}